\documentclass[11pt, reqno]{amsart}

\usepackage{amscd}
\usepackage{amsmath, amssymb}
\usepackage{hyperref,amsfonts,color}

\renewcommand{\leq}{\leqslant}
\renewcommand{\geq}{\geqslant}

\newcommand{\Z}{\mathbb{Z}}

\newcommand{\R}{\mathbb{R}}
\newcommand{\C}{\mathbb{C}}

\newtheorem{theorem}{Theorem}[section]

\newtheorem{proposition}[theorem]{Proposition}

\newcounter{mtheorem}
\newtheorem{mtheorem}[mtheorem]{Theorem}

\numberwithin{equation}{section}

\theoremstyle{definition}
\newtheorem{rk}[theorem]{Remark}
\newtheorem{ex}[theorem]{Example}
\newtheorem{step}{Step}

\begin{document}

\title{A Liouville theorem for the complex \\ Monge-Amp\`ere equation on product manifolds}
\begin{abstract}
Let $Y$ be a closed Calabi-Yau manifold. Let $\omega$ be the K\"ahler form of a Ricci-flat K\"ahler metric on $\C^m \times Y$. We prove that if $\omega$ is uniformly bounded above and below by constant multiples of  $\omega_{\C^m} + \omega_Y$, where $\omega_{\C^m}$ is the standard flat K\"ahler form on $\C^m$ and $\omega_Y$ is any K\"ahler form on $Y$, then $\omega$ is actually equal to a product K\"ahler form, up to a certain automorphism of $\C^m \times Y$.
\end{abstract}
\author{Hans-Joachim Hein}
\address{Department of Mathematics, Fordham University, Bronx, NY 10458, USA}
\email{hhein@fordham.edu}
\date{\today}
\thanks{Research partially supported by NSF grant DMS-1514709.}
\maketitle
\thispagestyle{empty}
\markboth{A Liouville theorem for the complex Monge-Amp\`ere equation on product manifolds}{Hans-Joachim Hein}

\section{Introduction}\label{s:intro}

The purpose of this article is to prove the following theorem. 

\begin{mtheorem}\label{liouville} Let $Y$ be a closed K\"ahler manifold with $c_1(Y)= 0$ in $H^2(Y,\R)$. Let $\omega$ be a Ricci-flat K\"ahler form on $\C^m \times Y$. Assume that $C^{-1}(\omega_{\C^m} + \omega_Y) \leq \omega \leq C(\omega_{\C^m} + \omega_Y)$ for some $C \geq 1$, where $\omega_{\C^m}$ is the standard flat K\"ahler form on $\C^m$ and $\omega_Y$ is any K\"ahler form on $Y$. Then:

{\rm (1)} There is a unique choice of $\omega_Y$ such that ${\rm Ric}(\omega_Y) = 0$ and $\omega$ is $d$-cohomologous to $\omega_{\C^m} + \omega_Y$. Fix this choice of $\omega_Y$. Then $\omega^{m+n}$ is a constant multiple of $\omega_{\C^m}^m \wedge \omega_Y^n$, where $n = \dim_\C Y$.  

{\rm (2)} There is a unique complex linear map $\ell: \C^m \to H^{0,1}(Y)$ such that $T_\ell^*\omega$ is $i\partial\bar\partial$-cohomologous to $\omega_{\C^m} + \omega_Y$, where $T_\ell$ denotes the automorphism of $\C^m \times Y$ defined in \eqref{e:trans} below.

{\rm (3)} There is a matrix $S \in {\rm GL}(m,\C)$ such that $T_\ell^*\omega = S^*\omega_{\C^m} + \omega_Y $.
\end{mtheorem}

Here (1) is obvious from Yau's theorem \cite{Yau} and the fact that $\log(\omega^{m+n}/\omega_{\C^m}^m \wedge \omega_Y^n)$ is a bounded pluriharmonic function on $\C^m \times Y$, hence on $\C^m$, if ${\rm Ric}(\omega_Y) = 0$. For (2), we identify $H^{0,1}(Y)$ with the space of parallel $(0,1)$-forms with respect to the Ricci-flat metric $\omega_Y$ \cite[Thm 6.11]{ball}. Then for all $\xi \in H^{0,1}(Y)$ we have a unique $\omega_Y$-parallel $(1,0)$-vector field $\xi^\sharp$ with $\xi^\sharp \,\lrcorner\,\omega_Y = \xi$. The map $\xi \mapsto \xi^\sharp$ is complex linear. Formally write the time-$1$ map of $\xi^\sharp$ on $Y$ as  $y \mapsto y + \xi^\sharp$ and define
\begin{align}\label{e:trans}
T_\ell(z,y) = (z, y + \ell(z)^\sharp).
\end{align}
Thus, (2) is trivial if $Y$ is a point. In this case, (3) reduces to a classical Liouville theorem for the complex Monge-Amp\`ere equation $\det(\phi_{z^i\bar{z}^j}) = 1$ for a plurisubharmonic function $\phi:\C^m\to\R$ \cite{RS}. This Liouville theorem is a consequence of the Calabi-Yau $C^3$ estimate \cite{Cal,Yau}. If $Y$ is not a point but is \emph{flat}, then this classical method still shows that $\omega$ is flat, and (2) and (3) follow from this. If $Y$ is \emph{not} flat, then this method breaks down due to the appearance of negative curvature terms in a Weitzenb\"ock type computation. So one way to think of Theorem \ref{liouville} is as a ``vanishing theorem'' for ``nonlinear harmonic $2$-forms'' that cannot be proved using the Weitzenb\"ock method.

It is a standard theme in analysis that conversely, Liouville theorems imply regularity theorems, e.g.~see \cite{CW}. The special case of Theorem \ref{liouville} where $\ell = 0$ and $\omega|_{\{z\}\times Y} = \omega_Y$ for all $z \in \C^m$ (a very strong assumption, which quickly reduces the claim to the classical case $n = 0$) was used in \cite{TZ} to establish higher regularity properties of collapsing Ricci-flat K\"ahler metrics and K\"ahler-Ricci flows on holomorphic fiber spaces with generic fiber $Y$. In \cite{TZ}, the assumption that $\omega|_{\{z\}\times Y} = \omega_Y$ for all $z$ was justified by the deep technical results of \cite{TWY}. It is thus worth noting that our work here does not rely on \cite{TWY}. New applications of Theorem \ref{liouville} to regularity problems will appear in \cite{HT}.

Here is a sketch of our proof of Theorem \ref{liouville}. For any constant vector field $\partial_{z}$ on $\C^m$, the function $|\partial_z|_\omega^2$ is $\omega$-subharmonic, and is constant if and only $\partial_z$ is $\omega$-parallel, thanks to the Bochner formula for the $\bar\partial$-operator on vector fields \cite[Prop 4.79]{ball} and to the fact that ${\rm Ric}(\omega) \leq 0$. Motivated by \cite[\S 2]{litam}, V.~Tosatti and Y.~Zhang (unpublished) observed that if $m = 1$, then this idea suffices to prove Theorem \ref{liouville} because a bounded subharmonic function on a parabolic manifold is constant.

To improve this strategy, we will first use the standard case $n = 0$ of Theorem \ref{liouville} to prove that the tangent cone at infinity of $(\C^m \times Y, \omega)$ is $(\C^m, \omega_{\C^m})$. Estimates for subharmonic functions then allow us to prove that $|\partial_z|_\omega^2$ is $\leq \frac{1}{2}$ globally (and is asymptotic to $\frac{1}{2}$ along almost every ray in $\C^m$). For $m > 1$ this is still consistent with the behavior of nonconstant bounded subharmonic functions such as $\max\{0,1-|z|^{2-2m}\}$. However, it turns out that essentially the same conclusions hold for the dual quantity $\sum_{j=1}^m |dz^j|_\omega^2 = 2{\rm tr}_{\omega}(\omega_{\C^m}) \leq 2m$, which is $\omega$-subharmonic because of the reverse inequality ${\rm Ric}(\omega) \geq 0$ \cite[Ex 1.38]{ball}. (The details are more involved because the metric on the cotangent bundle diverges to infinity in the $Y$-directions when we blow down.) Then it suffices to note  that
\begin{align}\label{cs}
m^2 = \biggl(\sum_{j=1}^m dz^j(\partial_{z^j})\biggr)^2 \leq \biggl(\sum_{j=1}^m |dz^j|_\omega \cdot  |\partial_{z^j}|_\omega\biggr)^2 \leq \biggl(\sum_{j=1}^m |dz^j|_\omega^2\biggr)\biggl(\sum_{j=1}^m |\partial_{z^j}|_\omega^2\biggr) \leq m^2.
\end{align}
This forces equality in our Bochner formulas, so that $\nabla^\omega\partial_{z^j} = 0$ and  $\nabla^\omega dz^j = 0$ for all $j$. Moreover, $|dz^j|_\omega^2 = \lambda|\partial_{z^j}|_\omega^2 = \frac{\lambda}{2}$ for some $\lambda \geq 0$, and hence $|dz^j|_\omega^2 = 2$ because $\sum_{j=1}^m |dz^j|_\omega^2 = 2m$.

The remainder of this paper is organized as follows. In Section \ref{s:outline} we give a more detailed outline of the proof of Theorem \ref{liouville} as a sequence of six steps. Some of these steps are obvious or follow from the literature. We provide technical details for the remaining ones in Section \ref{s:details}.

I would like to thank Bianca Santoro and Valentino Tosatti for very useful discussions.

\section{Outline}\label{s:outline}

\begin{step}\label{1}
Form the $1$-parameter family $\omega_t = \Psi_t^*(e^{-t}\omega)$ of Ricci-flat K\"ahler metrics on $\C^m \times Y$, where $\Psi_t(z,y) = (e^{t/2}z,y)$. As $t \to \infty$ this family connects $(\C^m\times Y, \omega)$ to its tangent cone at infinity. By assumption, $\omega_t$ is uniformly bounded above and below by fixed multiples of the Ricci-flat product metric $\omega_{\C^m} + e^{-t}\omega_Y$, and $\omega_t^{m+n}$ is equal to a fixed multiple of $e^{-nt}\omega_{\C^m}^m \wedge \omega_Y^n$. 
\end{step}

\begin{step}\label{2}
Use ideas from \cite{T} to find a subsequence $\omega_{t_i}$ that converges in the $C^{1,\alpha}_{loc}$ topology of K\"ahler potentials and in the weak topology of currents on $\C^m \times Y$ to ${\rm pr}_{\C^m}^*(\omega_\infty)$, where $\omega_{\infty}$ is a Bedford-Taylor solution to a complex Monge-Amp\`ere equation $\omega_\infty^m = c \omega_{\C^m}^m$ on $\C^m$ (here $c > 0$ is a constant). Since our set-up differs slightly from the one of \cite{T}, some preliminary steps are required.
\end{step}

\subsubsection*{Step 2.1} Prove that $\omega$ differs from $\omega_{\C^m} + \omega_Y$ by $i\partial\bar\partial$ of some global K\"ahler potential on $\C^m \times Y$, after ``translating'' $\omega$ by some automorphism $T_\ell$ as in \eqref{e:trans} if necessary (which does not affect the uniform equivalence of the two K\"ahler forms). In fact this is part of the content of Theorem \ref{liouville}(2). From this step on, we will assume throughout the whole paper that $T_\ell = {\rm id}$.

\subsubsection*{Step 2.2} Prove that on $B_1 \times Y$, where $B_1 = B_1(0)$ is the standard unit ball in $\C^m$, the Ricci-flat metric $\omega_t$ differs from the model metric $\omega_{\C^m} + e^{-t}\omega_Y$ by $i\partial\bar\partial \phi_t$ with $|\phi_t| \leq C$. 

\subsubsection*{Step 2.3} Use ideas from \cite{T} to extract a sequence $\omega_{t_i}$ as described above. Thanks to Step \ref{2}.2, this is a direct adaptation (and simplification) of the proof of \cite[Thm 4.1]{T}, so we omit the details. We only remark that unlike in \cite{T}, it is not obvious at this point that the limit $\omega_\infty$ will be independent of the sequence $t_i$, but for us it suffices to have convergence along one fixed sequence.

\begin{step}\label{3}
By construction, the Bedford-Taylor solution $\omega_\infty$ to $\omega_\infty^m = c\omega_{\C^m}^m$ produced as a weak limit of a collapsing sequence $\omega_{t_i}$ in Step \ref{2} satisfies $C^{-1}\omega_{\C^m} \leq \omega_\infty \leq C \omega_{\C^m}$ in the sense of currents on $\C^m$. By \cite[Thm 1.7]{W} (see \cite[Prop 3.13]{W} for the fact that the notions of a Bedford-Taylor and a viscosity solution are equivalent), whose proof made extensive use of Caffarelli's work as in \cite{CC}, $\omega_\infty$ is then H\"older continuous and hence, by standard bootstrapping arguments, smooth and uniformly equivalent to $\omega_{\C^m}$. The standard Liouville theorem for the complex Monge-Amp{\`e}re equation of \cite{RS} now tells us that $\omega_\infty = S^*\omega_{\C^m}$ for some matrix $S \in {\rm GL}(m,\C)$. Thus, $\phi_{t_i} \to (S^*-{\rm id}^*)(|z|^2/2) + \pi$ in $C^{1,\alpha}_{loc}(\C^m\times Y)$ for some pluriharmonic function $\pi$ on $\C^m$. Replacing $\omega_t$ by $(S^{-1})^*(\omega_t)$ and $\phi_{t}$ by $(S^{-1})^*(\phi_{t}-\pi) + ((S^{-1})^* - {\rm id}^*)(|z|^2/2)$, we may then assume throughout the rest of the paper that $S = {\rm id}$, and in fact $\omega_{t} = (\omega_{\C^m} + e^{-t}\omega_Y) + i\partial\bar\partial\phi_t$ with $\phi_{t_i} \to 0$ in $C^{1,\alpha}_{loc}(\C^m \times Y)$.
\end{step}

\begin{step}\label{4}
It follows trivially from Step \ref{3} that for any constant holomorphic vector field $\partial_z$ on $\C^m$, the component function $|\partial_z|$\begin{footnotesize}$_{\omega_{t_i}}^2$\end{footnotesize} converges to the constant $\frac{1}{2}$ in the sense of distributions on $\C^m \times Y$. On the other hand, the dual quantity $|dz|$\begin{footnotesize}$_{\omega_{t_i}}^2$\end{footnotesize}, while uniformly bounded, does not obviously converge to anything (even in the sense of distributions) because the metric on the cotangent bundle induced by $\omega_t$ diverges in the fiber directions. To fix this problem, we will prove that the {\bf sum} $\sum_{j=1}^m |dz^j|$\begin{footnotesize}$_{\omega_{t_i}}^2$\end{footnotesize} converges to $2m$ in the sense of distributions, {\bf provided} that we restrict ourselves to test functions pulled back from the base. This strongly uses the convergence $\phi_{t_i}\to 0$ at the potential level.
\end{step}

\begin{step}\label{5} We now promote the distributional convergence $|\partial_z|$\begin{footnotesize}$_{\omega_{t_i}}^2$\end{footnotesize}$\to\frac{1}{2}$ and $\sum_{j=1}^m |dz^j|$\begin{footnotesize}$_{\omega_{t_i}}^2$\end{footnotesize}$\to 2m$ of Step \ref{4} to stronger forms of convergence, using the fact that both of these functions are $\omega_{t_i}$-subharmonic (which has not played any role so far). To this end, we first adapt a standard property of sequences of subharmonic functions on a ball in Euclidean space \cite[Thm 3.2.12]{H} to prove that:
\end{step}

\subsubsection*{Step 5.1} We have $|\partial_z|$\begin{footnotesize}$_{\omega_{t_i}}^2$\end{footnotesize}$\to\frac{1}{2}$ and $\sum_{j=1}^m |dz^j|$\begin{footnotesize}$_{\omega_{t_i}}^2$\end{footnotesize}$\to 2m$ in $L^1_{loc}(\C^m \times Y)$. This proof also relies directly on the fact that ${\rm Ric}(\omega_{t_i}) \geq 0$ via an application of Li-Yau type heat kernel estimates.

\vskip2mm

Using subharmonicity once again gives us pointwise bounds. This is similar to \cite[Thm 3.2.13]{H}, but as in Step \ref{5}.1 we need to replace the Euclidean methods of \cite{H} by a heat kernel estimate.

\subsubsection*{Step 5.2} We have $\sup_{B_1 \times Y} |\partial_z|$\begin{footnotesize}$_{\omega_{t_i}}^2$\end{footnotesize}$\leq \frac{1}{2} + o(1)$ and $\sup_{B_1 \times Y} \sum_{j=1}^m |dz^j|$\begin{footnotesize}$_{\omega_{t_i}}^2$\end{footnotesize}$\leq 2m + o(1)$ as $i \to \infty$. 

\vskip2mm

The second statement is similar to a trace estimate \cite[Lemma 4.7]{TWY} that played a crucial role in \cite{TWY} in a different setting and was quite difficult to prove there (using different methods).

\begin{step}\label{6}
Undoing the scaling and stretching of Step \ref{1}, the result of Step \ref{5} says that on $(\C^m \times Y,\omega)$, the subharmonic functions $|\partial_z|_\omega^2$ and $\sum_{j=1}^m |dz^j|_\omega^2$ are $\leq \frac{1}{2}$ and $\leq 2m$ globally (and are asymptotic to these constants along almost every ray in $\C^m$). Combining this with \eqref{cs}, we deduce that they are constant equal to $\frac{1}{2}$ and $2m$, respectively. The equality case in the Bochner formula now shows that $\partial_{z^j}$ and $dz^j$ are $\omega$-parallel for all $j$. It follows from a Riemannian argument, together with the fact that $\omega$ is $i\partial\bar\partial$-cohomologous to $\omega_{\C^m} + \omega_Y$, that $\omega = \omega_{\C^m} + \omega_Y$, proving Theorem \ref{liouville}(3).
\end{step}

\section{Proofs}\label{s:details}

Step \ref{1} is self-explanatory, and details for Steps \ref{2}.3 and \ref{3} are available from the literature. Thus, it remains to prove the claims made in Steps \ref{2}.1, \ref{2}.2, and \ref{4}, \ref{5}, \ref{6}. All notation and conventions of
Sections \ref{s:intro}--\ref{s:outline} will remain in force throughout this section. Because of this, some results actually hold in slightly greater generality than stated. For example, in Proposition \ref{p:translate}, the $\omega_{\C^m}$ term and the fact that ${\rm Ric}(\omega) = 0$ play no role, but our standing assumptions that $\omega - (\omega_{\C^m} + \omega_Y)$ is $d$-exact (by the conventions of Theorem \ref{liouville}(1)) and that ${\rm Ric}(\omega_Y) = 0$ are crucial (although the latter could be relaxed to the conclusion of the Beauville-Bogomolov-Calabi decomposition theorem).

\subsection{K\"ahler potentials} The following propositions deal with Steps \ref{2}.1 and \ref{2}.2, respectively.

\begin{proposition}\label{p:translate}
There exists a unique complex linear map $\ell: \C^m \to H^{0,1}(Y)$ such that $T_\ell^*\omega$ is $i\partial\bar\partial$-cohomologous to $\omega_{\C^m} + \omega_Y$, where $T_\ell$ denotes the automorphism of $\C^m \times Y$ defined in \eqref{e:trans}.
\end{proposition}

\begin{proof}
During this proof we will abbreviate  $ \omega_{\C^m} + \omega_Y = \omega_P$ and $g_{\C^m} + g_Y = g_P$.

By assumption $\zeta = \omega - \omega_P$ is exact. Let $\xi$ be any real $1$-form with $d\xi = \zeta$. Then $\bar\partial \xi^{0,1} = 0$, so $\xi^{0,1}$ defines a class in $H^{0,1}(\C^m \times Y)$. By the Leray spectral sequence of the projection $\C^m \times Y \to \C^m$, there is an isomorphism $\Phi: H^{0,1}(\C^m \times Y) \to \mathcal{O}(\C^m, H^{0,1}(Y))$ with $\Phi[\xi^{0,1}](z) = [\xi^{0,1}|_{\{z\}\times Y}]$. This is made more explicit in the proof of \cite[Lemma 4.1]{GW} if $\dim_\C Y = 1$. The Beauville-Bogomolov-Calabi decomposition theorem \cite[Thm 6.6]{ball} allows us to generalize the computations of \cite{GW} to our setting. Writing $\Phi[\xi^{0,1}] = f$ and identifying $H^{0,1}(Y)$ with the space of $g_Y$-parallel $(0,1)$-forms, $(\partial_{\bar{z}^j}f)(z) = 0$ by \cite[p.515]{GW}, and similar arguments then show that $(\partial_{z^j}f)(z)$ is the $H^{0,1}$-class, or the parallel part with respect to $g_Y$, of the $(0,1)$-form $(\partial_{z^j} \,\lrcorner\,\zeta)|_{\{z\}\times Y}$ (independent of our choice of $\xi$). Since this form is uniformly bounded with respect to $g_P$, we learn that $f$ is a holomorphic polynomial of degree $\leq 1$. Changing $\xi$ by a suitable $g_P$-parallel form, we may assume that $f(0) = 0$.

Writing $T_\ell^*\omega - \omega_P = T_\ell^*(\zeta + \tilde{\zeta}_\ell)$ with $\tilde\zeta_\ell = \omega_P - T_{-\ell}^*\omega_P$, and arguing as in the usual proof of the $i\partial\bar\partial$-lemma, one now easily checks that our proposition is equivalent to the following claim.\medskip\

\noindent \emph{Claim}. There exists a unique complex linear map $\ell: \C^m \to H^{0,1}(Y)$ such that there exists a real $1$- form $\tilde\xi$ such that $d\tilde\xi = \tilde\zeta_\ell$ and $\Phi[\tilde\xi^{0,1}] = -f$.\medskip\

To prove uniqueness, notice that by the same reasoning as above, $\Phi[\tilde\xi^{0,1}]$ must be a holomorphic polynomial of degree $\leq 1$ with $1$-homogeneous part equal to $\ell$, so that necessarily $\ell = -f$. 
To prove that this choice of $\ell$ works, we first compute as in \cite[p.382]{Hein} that $\tilde\zeta_\ell = d\tilde\xi$ with 
$$\tilde\xi = \ell^\sharp \,\lrcorner \int_0^1 T_{-t\ell}^*\omega_P \, dt.$$
Then it is easy to see that this choice of $\tilde\xi$ satisfies $\Phi[\tilde{\xi}^{0,1}] = \ell = -f$ as desired.
\end{proof}

In the following proposition and in its proof, we fix $\omega_{\C^m}$ and $\omega_Y$ as reference metrics on $\C^m$ and $Y$. Also note that we will apply this proposition to $\zeta = \zeta_t =  \omega_t - (\omega_{\C^m} + e^{-t}\omega_Y)$.

\begin{proposition}
There is a constant $C$ such that if $\zeta$ is an $i\partial\bar\partial$-exact $(1,1)$-form on $B_2 \times Y$, then there is a function $\phi$ on $B_2 \times Y$ with $i\partial\bar\partial\phi = \zeta$ and $\|\phi\|_{L^\infty(B_1\times Y)} \leq C\|\zeta\|_{L^\infty(B_2 \times Y)}$.
\end{proposition}

\begin{proof}
%A subscript $L^p$ indicates the $L^p$ norm over $B_2 \times Y$ or $B_2$.

By assumption there is a potential $\psi$ on $B_2 \times Y$ such that $\zeta=i\partial\bar\partial\psi$. Let $\underline{\psi}$ be the fiberwise average of $\psi$. This is a function on $B_2$. Writing $\underline\zeta = i\partial\bar\partial\underline\psi$,  \cite[(4.9)]{T} shows that  $\|\underline\zeta\|_{L^\infty} \leq C\|\zeta\|_{L^\infty}$. The usual proof of the Poincar\'e lemma yields a $1$-form $\underline\xi$ on $B_2$ with $d\underline\xi = \underline\zeta$ and $\|\underline\xi\|_{L^\infty} \leq C\|\underline\zeta\|_{L^\infty}$. Then $\bar\partial\underline\xi^{0,1} = 0$, so $\bar\partial$-Neumann theory \cite[Cor 8.10]{Dem} produces a weak solution $\underline\psi'$ to $\bar\partial\underline\psi' = \underline\xi^{0,1}$ with $\|\underline\psi'\|_{L^2} \leq C\|\underline\xi^{0,1}\|_{L^2}$. Defining $\underline\psi'' = 2 {\rm Im}\,\underline\psi'$, it follows that $i\partial\bar\partial\underline\psi'' = \underline\zeta$ as currents on $B_2$, hence in particular $\Delta\underline\psi'' = {\rm tr}(\underline\zeta)$ as distributions, so $\underline\psi''$ is smooth with $\|\underline\psi''\|_{L^\infty(B_1)} \leq C(\|\underline\psi''\|_{L^2} + \|\underline\zeta\|_{L^\infty})$. Thus, $\zeta = i\partial\bar\partial(\psi-\underline\psi + \underline\psi'')$ with $\|\underline\psi''\|_{L^\infty(B_1)} \leq C\|\underline\zeta\|_{L^\infty}$. Then it remains to note that $\psi-\underline\psi$ has mean value zero on each fiber and $\Delta_Y(\psi-\underline\psi) = {\rm tr}(\zeta|_Y)$, so that $\|\psi-\underline\psi\|_{L^\infty} \leq C\|\zeta\|_{L^\infty}$ by fiberwise Moser iteration. Thus, a potential $\phi$ with the desired properties is given by $\phi = \psi-\underline\psi + \underline\psi''$.
\end{proof}

\subsection{Weak convergence of the functions $\sum_{j = 1}^m |dz^j|_{\omega_{t_i}}^2$} Proposition \ref{weaktrace} carries out Step \ref{4}. This is the first key difficulty of this paper. Let us first see why this step is nontrivial.

\begin{ex}\label{bad_example}
Define a flat K\"ahler metric $\omega_t$ on $\C^2$ with coordinates $(w^1,w^2) = (z,y)$ by
\begin{align}\label{e:bad_example}
\omega_t = \frac{i}{2}g_{j\bar{k}}dw^j\wedge d\bar{w}^k, \;\,(g_{j\bar{k}}) = \begin{pmatrix} 1 & c e^{-t/2}\\  c e^{-t/2}& (c^2+1)e^{-t} \end{pmatrix},
\end{align}
for a fixed $c\in\mathbb{R}$. This induces a flat K\"ahler metric on $\C \times Y$, where $Y = \C/(\Z + \Z i)$. Notice that the metric $\omega_Y$ in the sense of Theorem \ref{liouville}(1) is given by $\omega_Y$ $=$ $\frac{i}{2}(c^2+1)dy \wedge d\bar{y}$. Then $\omega_t$ is uniformly bounded above and below by the product metric $\omega_{\C} + e^{-t}\omega_Y$ (where $\omega_\C = \frac{i}{2}dz \wedge d\bar{z}$), has the same determinant as this product metric up to a fixed constant factor, and converges to $\omega_\C$ with respect to any fixed norm as $t \to \infty$. But ${\rm tr}_{\omega_t}(\omega_\C) = \frac{1}{2}|dz|_{\omega_t}^2 =c^2+1 > m = 1$ unless $c = 0$.
\end{ex}

It is then clear that in order to prove that the limit of ${\rm tr}_{\omega_{t_i}}(\omega_{\C^m})$ is $m$ in our situation, we need to find a way of proving that the ``off-diagonal terms'' as in \eqref{e:bad_example} are negligible. The next proposition achieves this in a weak sense, using the fact that the difference of $\omega_t$ and $\omega_{\C^m} + e^{-t}\omega_Y$ is $i\partial\bar\partial$-exact (after Step \ref{2}.1). This fails to be true in Example \ref{bad_example} for $c \neq 0$. More precisely:

\begin{ex}\label{bad_example_contd}
In Example \ref{bad_example}, the off-diagonal terms of $\omega_t - (\omega_\C + e^{-t}\omega_Y)$ are given by $ce^{-t/2}\zeta$, where $\zeta = \frac{i}{2}(dz \wedge d\bar{y} + dy \wedge d\bar{z})$. On the universal cover, $\zeta = d\xi$ with $\xi = {\rm Im}(\bar{z}dy)$, and $\xi$ is invariant under the deck group. Also on the universal cover, $\zeta = i\partial\bar\partial \phi$ with $\phi = {\rm Re}(\bar{z}{y})$, but there exists no pluriharmonic function $\pi$ on $\C^2$ such that $\phi +\pi$ is invariant under the deck group; indeed, if $\pi$ did exist, then $\phi + \pi$ would depend only on $z$ by Liouville's theorem applied fiberwise, contradicting the fact that $i\partial\bar\partial(\phi+\pi) = \zeta$, which contains $dy$ and $d\bar{y}$ components.

It is instructive to compare this with the formalism from \cite{GW} used in the proof of Proposition \ref{p:translate}.
With $\xi$ as above, we have that $\xi^{0,1}|_{\{z\} \times Y} = \frac{i}{2}z d\bar{y}$, which is parallel on $\{z\}\times Y$, so the holomorphic function $f = \Phi[\xi^{0,1}]$ is given by $f(z) = \frac{i}{2}z d\bar{y}$. Moreover, if $T_t(z,y) = (z,y-Az)$ with $A = \frac{c}{c^2+1}e^{t/2}$, then $\tilde\omega_t = T_t^*\omega_t$ has fundamental matrix ${\rm diag}(\frac{1}{c^2+1}, (c^2+1)e^{-t})$ and ${\rm tr}_{\tilde{\omega}_t}(\tilde{\omega}_\infty) = 1$ as desired.
\end{ex}

\begin{proposition}\label{weaktrace}
For every test function $\eta \in C^\infty_0(\C^m)$ on the base,
\begin{align}\label{e:weaktrace}
\lim_{i\to\infty} \int_{\C^m \times Y} \eta({\rm tr}_{\omega_{t_i}}(\omega_{\C^m}) - m)(\omega_{\C^m}^m \wedge \omega_Y^n) = 0.
\end{align}
\end{proposition}

\begin{proof}
We abbreviate $\omega_{\C^m} + e^{-t}\omega_Y =  \omega_{P,t}$, so that $\omega_t = \omega_{P,t} + i\partial\bar\partial\phi_t$ by Steps \ref{2}--\ref{3} with $\phi_{t_i} \to 0$ in $C^{1,\alpha}_{loc}(\C^m \times Y)$ for a certain sequence $t_i \to \infty$, which is the sequence featuring in \eqref{e:weaktrace}. For simplicity we will pretend that $\phi_t \to 0$ as $t\to\infty$. The following arguments are inspired by the method of proof of \cite[Thm 4.1]{T} (recall that we already used a variant of the latter in Step \ref{2}.3).

Thanks to the fact that $e^{nt}\omega_t^{m+n} =  c(\omega_{\C^m}^m \wedge \omega_Y^n)$, it suffices to consider the expression
\begin{align}
&[{\rm tr}_{\omega_t}(\omega_{\C^m}) - m] (e^{nt}\omega_{t}^{m+n})\notag\\
&=e^{nt}[(m+n)\omega_{\C^m} \wedge (\omega_{P,t} + i\partial\bar\partial\phi_t)^{m+n-1} - m(\omega_{P,t} + i\partial\bar\partial\phi_t)^{m+n}]\notag \\
&=\sum_{k=0}^{m+n} e^{nt}\left[(m+n) {m+n-1\choose k}\omega_{\C^m} \wedge \omega_{P,t}^{m+n-1-k} - m{m+n\choose k}\omega_{P,t}^{m+n-k}\right]\wedge (i\partial\bar\partial\phi_t)^k\notag\\
&=\sum_{k=0}^{m+n}\left[\sum_{\ell = \max\{0,m-k\}}^{\min\{m,m+n-k\}} e^{(\ell-m+k)t} \frac{(m+n)!(\ell-m)}{k!\ell! (m+n-k-\ell)!}\omega_{\C^m}^\ell \wedge \omega_Y^{m+n-k-\ell}\right]\wedge (i\partial\bar\partial\phi_t)^k.\label{e:expand}
\end{align}
We now decompose $\phi_t = \underline{\phi_t} + (\phi_t - \underline{\phi_t})$, where the underline denotes fiberwise averages, and further $i\partial\bar\partial(\phi_t-\underline{\phi_t}) = \alpha_t+\beta_t+\gamma_t$ according to $\Lambda^2(\C^m\times Y) = \Lambda^2\C^m \oplus (\Lambda^1\C^m \otimes \Lambda^1Y) \oplus \Lambda^2Y$. Then
\begin{align}
\label{e:expand2}
(i\partial\bar\partial \phi_t)^k &= \sum_{p=0}^k {k\choose p}(i\partial\bar\partial\underline{\phi_t})^p \wedge (i\partial\bar\partial(\phi_t-\underline{\phi_t}))^{k-p}\\
&= \sum_{p+q+r+s=k} \frac{k!}{p!q!r!s!}(i\partial\bar\partial\underline{\phi_t})^p \wedge \alpha_t^q \wedge \beta_t^r \wedge \gamma_t^s.\label{e:expand3}
\end{align}
Since $|i\partial\bar\partial\phi_t| \leq C$, the same holds for $i\partial\bar\partial\underline{\phi_t}$ by \cite[(4.9)]{T}, and hence for $\alpha_t,\beta_t,\gamma_t$. However, because $\omega_t|_Y = e^{-t}\omega_Y + \gamma_t$, we actually know that $|\gamma_t| \leq Ce^{-t}$ with respect to any fixed background metric. Similarly, $|\beta_t| \leq Ce^{-t/2}$ by a Cauchy-Schwarz estimate with respect to $\omega_t$. As $2(\ell+p+q)+r = 2m$ for all nonzero terms in the big sum over $k,\ell,p,q,r,s$, we get $s + \frac{r}{2} =\ell-m+k$, so these bounds on $|\gamma_t|$ and $|\beta_t|$ compensate the $e^{(\ell-m+k)t}$ factors in \eqref{e:expand}, leading to an $O(1)$ overall bound.

Now even if $i\partial\bar\partial\phi_t \to 0$ in $L^{\mathbf p}_{loc}$ for all $\mathbf p$ (which is more than we know from  Step \ref{3}), we would not be able to prove that the terms with $p = q = 0$ and $r > 0$ are $o(1)$ rather than $O(1)$. Indeed, so far we have not made any real use of the fact that $\omega_t - \omega_{P,t}$ is $i\partial\bar\partial$-exact, and the corresponding terms in Example \ref{bad_example} are precisely the ones that obstruct the desired convergence. To fix this problem, we now multiply the whole expression by $\eta$ and integrate by parts, as follows. 

$\bullet$ If $p < k$ in \eqref{e:expand2}, throw one $i\partial\bar\partial$ from $(i\partial\bar\partial(\phi_t-\underline{\phi_t}))^{k-p}$ onto $\eta$, then proceed as in \eqref{e:expand3} with $k$ replaced by $k-1$. In the resulting sum, $2(\ell + p+q+1) + r = 2m$ for all nonzero terms because of the additional factor of $i\partial\bar\partial\eta$, so again $s + \frac{r}{2} = \ell-m+k$ as before. Then it suffices to observe that $|\phi_t - \underline{\phi_t}| = O(e^{-t})$ pointwise because $0 \leq ne^{-t} + \Delta_Y(\phi_t - \underline{\phi_t}) = {\rm tr}(\omega_t|_Y) \leq Ce^{-t}$.

$\bullet$ If $0 < p = k$ in \eqref{e:expand2}, then all nonzero terms in the remaining sum over $k,\ell$ satisfy $k+\ell = m$, so the $e^{(\ell-m+k)t}$ factor in \eqref{e:expand} is bounded. Throw one $i\partial\bar\partial$ from $(i\partial\bar\partial\underline{\phi_t})^k$ onto $\eta$ and observe that $\underline{\phi_t} \to 0$ locally uniformly on $\C^m \times Y$ because the same is true for $\phi_t$.

$\bullet$ If $0 = p = k$ in \eqref{e:expand2}, then we simply note that the $k = 0$ term in \eqref{e:expand} is zero anyway.
\end{proof}

\subsection{From weak to $L^1_{loc}$ and pointwise} Proposition \ref{l1loc} below deals with Step \ref{5}.1, proving that the weak convergence of Step \ref{4} can be promoted to $L^1_{loc}$ convergence. This is the second technical difficulty of this paper after Step \ref{4}. Proposition \ref{p:upper_bounds} then establishes Step \ref{5}.2, which is easier.

\begin{proposition}\label{l1loc}
Let $u_i$ be a sequence of smooth $\omega_{t_i}$-subharmonic functions on $\C^m \times Y$ such that $|u_i| \leq C$. Assume that there exists a constant function $u$ such that for all $\eta \in C^\infty_0(\C^m)$, 
\begin{align}\label{hyp:weak}
\lim_{i\to\infty} \int_{\C^m \times Y} \eta (u_i-u)(\omega_{\C^m}^m \wedge \omega_Y^n) = 0.
\end{align} 
Then $u_i$ converges to $u$ in $L^1_{loc}$ with respect to the fixed volume form $\omega_{\C^m}^m \wedge \omega_Y^n$ on $\C^m \times Y$. 
\end{proposition}

\begin{proof} For subharmonic functions on $\R^d$ this is a special case of \cite[Thm 3.2.12]{H}. We will now adapt the proof of \cite{H} to our setting. Because the $u_i$ are smooth and converge to a constant function $u$ as distributions, we can safely skip the part of the proof before (3.2.5) in \cite{H}.

The main point of the remainder of the proof in \cite{H} is that if $v_i$ are smooth, uniformly bounded, subharmonic functions on $\R^d$ converging weakly to a smooth subharmonic function $v$ (not required to be constant), and if  $\varphi_\delta$ is a standard family of radial mollifiers, then:

(1) $v_i \leq \varphi_\delta \ast v_i$ (this is the only point where the subharmonicity of $v_i$ is used),

(2) the sequence $\varphi_\delta \ast v_i$ is locally equibounded and equicontinuous for any fixed $\delta$,

(3) $\varphi_\delta \ast v_i \to \varphi_\delta \ast v$ in $L^\infty_{loc}$ as $i \to\infty$, and $\varphi_\delta \ast v \to v$ in $L^1_{loc}$ as $\delta \to 0$.

To recover these properties in our setting, we now replace H\"ormander's smoothing operator $\varphi_\delta \, \ast$ by the time-$\delta^2$ heat evolution operator $H_{i,\delta} \,\ast$ associated with $(\C^m \times Y, \omega_{t_i})$. Then:

($1'$) $u_i \leq H_{i,\delta} \ast u_i$,

($2'$) $|H_{i,\delta} \ast u_i| \leq C$ and
$\sup_{\C^m \times Y} |\nabla^{\omega_{t_i}}(H_{i,\delta} \ast u_i)|_{\omega_{t_i}}  \leq C\min\{\delta,1\}^{-1}$, and 

($3'$) for all $\delta,R > 0$ it holds that $\lim_{i\to\infty} \sup_{B_R\times Y} |(H_{i,\delta}\ast u_{i}) - u| = 0$.

Let us prove these properties before showing that they imply the desired $L^1_{loc}$ convergence. ($1'$) and ($2'$) immediately follow from our assumption that $|u_i| \leq C$, together with suitable versions of the maximum principle \cite[Thm 15.2]{li} and of the Cheng-Yau gradient estimate \cite[Thm 3]{Kot}. Here we  rely on the fact that ${\rm Ric}(\omega_{t_i}) \geq 0$ to ensure that the constant $C$ of ($2'$) is uniform. 

To prove ($3'$), fix $\delta,R > 0$. Let $w_i = (H_{i,\delta} \ast u_i) - u$.
By ($2'$) and Arzel\`a-Ascoli, some subsequence $w_{i_j}$ converges in $C^{0,\alpha}$ on $B_R \times Y$ to $w_\infty$, where $w_\infty$ is Lipschitz, and is constant on all fibers because $|w_i(z,y) - w_i(z,y')| \leq C_\delta e^{-t_i/2}{\rm dist}_{\omega_{Y}}(y,y')$ by ($2'$). Thus, to prove that $w_\infty = 0$, it suffices to prove that $\int \chi w_{i_j}(\omega_{\C^m}^m \wedge \omega_Y^n) \to 0$ for any fixed test function $\chi \in C^\infty_0(B_R)$ on the base, after passing to a further subsequence depending on $\chi$ if necessary.\footnote{As it turns out, we will not actually need to use the fact that $w_\infty$ and $\chi$ are pulled back from the base. However, the argument proving that $w_\infty$ is pulled back from the base also applies to $\eta_\infty$ below, and this will be crucial.} We will now prove this. 

The key step is to verify the identity 
\begin{align}\label{e:associate}
\int_{\C^m \times Y} \chi w_{i_j}(\omega_{\C^m}^m \wedge \omega_Y^n) = \int_{\C^m \times Y} \eta_{i_j}(u_{i_j}-u)(\omega_{\C^m}^m\wedge\omega_Y^n), \;\, \eta_i = H_{i,\delta} \ast \chi.
\end{align}
For this one uses the fact that $w_i$ $=$ $H_{i,\delta} \ast (u_i-u)$ because $u$ is constant, the Monge-Amp\`ere equation $\omega_{t_i}^{m+n} = c e^{-nt_i}(\omega_{\C^m}^m \wedge \omega_Y^n)$ (here $c$ is some fixed normalizing factor), and standard Gaussian upper bounds for the heat kernel on a \emph{fixed} manifold in order to justify applying Fubini's theorem.

The estimates ($2'$) hold verbatim for $\eta_i$ (with constants depending on $\chi$), so applying Arzel\`a-Ascoli once again, we learn that some subsequence $\eta_{i_{j_k}}$ converges in $C^{0,\alpha}_{loc}$ on $\C^m \times Y$ to $\eta_\infty$, where $\eta_\infty$ is globally uniformly Lipschitz, and is constant on all fibers. In addition, by \cite[Thm 13.4]{li},
\begin{align}
|\eta_i(x)| & \leq C \int_{\C^m\times Y} |B_{\omega_{t_{i}}}(x,\delta)|_{\omega_{t_i}}^{-1/2}|B_{\omega_{t_{i}}}(x',\delta)|_{\omega_{t_i}}^{-1/2}\exp(-{\rm dist}_{\omega_{t_{i}}}(x,x')^2/5\delta^2)|\chi(x')| \omega_{t_{i}}(x')^{m+n}\notag\\
&\leq C_{\delta,R,\chi}\exp(-{\rm dist}_{\omega_{\C^m}}(z,B_R)^2/C\delta^2), \;\,x = (z,y) \in \C^m\times Y.\label{e:gaussian}
\end{align}

Now fix any $\epsilon>0$, aiming to prove that for all $k \geq k_{\delta,R,\chi,\epsilon}$, the absolute value of the left-hand side of \eqref{e:associate} for $j = j_k$ is bounded by some constant $C_{\delta,R,\chi}$ times $\epsilon$. To this end, decompose
\begin{align}\label{e:split}
\eta_i = \chi_{S} \eta_{\infty,\beta} + \chi_{S} (\eta_\infty - \eta_{\infty,\beta})+ \chi_{S}(\eta_i - \eta_\infty) + (1-\chi_{S})\eta_i, \end{align}
where $\eta_{\infty,\beta} = \varphi_{\beta} \ast \eta_\infty$ for some standard radial mollifier $\varphi_\beta$ of radius $\beta$ on $\C^{m}$ (with $\beta = \beta_{S,\epsilon}$ to be determined) and where $\chi_S$ is some standard smoothing of the characteristic function of $B_{S}$ in $\C^m$ (with $S = S_{\delta,R,\epsilon}$ to be determined). Notice that $|\eta_\infty - \eta_{\infty,\beta}| \leq C_{\delta,\chi}\beta$ globally on $\C^m$ because $\eta_\infty$ is uniformly Lipschitz. Also, if $k \geq k_{\delta,\chi,S,\gamma}$, then $|\eta_{i_{j_k}} - \eta_\infty| \leq \gamma$ on $B_{2S} \times Y$. Bringing these bounds and \eqref{e:gaussian} into \eqref{e:split} for $i = i_{j_k}$, and using \eqref{e:associate} for $j = j_k$, we learn that for all $k\geq k_{\delta,\chi,S,\gamma}$, 
\begin{align}
\left|\int_{B_R \times Y} \chi w_{i_{j_k}}(\omega_{\C^m}^m \wedge \omega_Y^n)\right| &\leq \left|\int_{\C^m\times Y} \chi_S\eta_{\infty,\beta}(u_{i_{j_k}}-u)(\omega_{\C^m}^m \wedge \omega_Y^n)\right|\label{final1}\\
&\hskip4.7mm + C_{\delta,\chi} S^{2m}\beta + C S^{2m}\gamma + C_{\delta,R,\chi} \int_{S/2}^\infty \exp(-\rho^2/C\delta^2)\rho^{2m-1}d\rho,\notag
\end{align}
provided that $S \geq 4R$. Since $\chi_S\eta_{\infty,\beta} \in C^\infty_0(\C^m)$, our hypothesis \eqref{hyp:weak} implies that the right-hand side of \eqref{final1} is at most $\epsilon$ once $k \geq k_{\delta,R,\chi,S,\beta,\epsilon}$. Thus, it remains to choose $S = S_{\delta,R,\epsilon}$ large enough (roughly on the order of $|{\log \epsilon}|^{1/2}$ if $\delta$ and $R$ are given) and $\beta = \beta_{S,\epsilon}$ and $\gamma = \gamma_{S,\epsilon}$ small enough.

The upshot  is that $\lim_{j\to\infty} \sup_{B_R \times Y} |(H_{i_j,\delta} \ast u_{i_j}) - u| = 0$ for some sequence $i_j$ depending on $\delta,R$. Finally, we note that passing to such a subsequence is actually unnecessary because what we have really proved here is that given $\delta$ and $R$, every \emph{subsequence} of $u_{i}$ has a further subsequence so that ($3'$) holds along this sub-subsequence; and this obviously implies ($3'$) as stated.

We are now in position to adapt the end of the proof of \cite[Thm 3.2.12]{H} to our setting. In fact, we can follow \cite{H} almost word by word. Choose $\eta \geq 0$ in $C^\infty_0(\C^m)$ and $\epsilon > 0$. By \eqref{hyp:weak},
$$\lim_{i\to\infty}\int_{\C^m\times Y}\eta(u - u_i + \epsilon)(\omega_{\C^m}^m \wedge \omega_Y^n) =  \int_{\C^m\times Y} \epsilon\eta(\omega_{\C^m}^m \wedge \omega_Y^n).$$
Choose $\delta > 0$, and choose $R > 0$ with ${\rm supp}(\eta) \subset B_R$. Then $u_i \leq H_{i,\delta} \ast u_i$ by ($1'$), and $H_{i,\delta}\ast u_i \to u$ uniformly on ${\rm supp}(\eta) \times Y$ by ($3'$). Thus, $\eta(u - u_i + \epsilon) \geq 0$ for $i \geq i_{\eta,\epsilon}$. This implies that
$$\limsup_{i\to\infty} \int_{\C^m\times Y} \eta|u_i-u|(\omega_{\C^m}^m \wedge \omega_Y^n) \leq 2\int_{\C^m\times Y} \epsilon\eta(\omega_{\C^m}^m \wedge \omega_Y^n)$$
by subtracting and adding $\epsilon$ on the left-hand side. The claim follows by letting $\epsilon\to 0$.
\end{proof}

\begin{rk}\label{rk:grig}
It seems likely that thanks to our assumption that $\omega_t$ is uniformly comparable to a collapsing product metric, the Gaussian upper heat kernel bounds used in \eqref{e:gaussian} (and in the proof of Proposition \ref{p:upper_bounds}) are true without any (Ricci) curvature conditions on $\omega_t$ or $\omega_Y$; see \cite[Cor 15.25]{Grig}. This seems less clear for the gradient estimates used in the proof of ($2'$) and ($3'$) above.
\end{rk} 

\begin{proposition}\label{p:upper_bounds}
Let $v_i$ be a sequence of smooth $\omega_{t_i}$-subharmonic functions on $\C^m \times Y$ such that $v_i^+ \leq C$. If $v_i^+ \to 0$ in $L^1_{loc}$ with respect to $\omega_{\C^m}^m \wedge \omega_Y^n$ on $\C^m\times Y$, then $\sup_{B_1 \times Y} v_i^+ \to 0$.
\end{proposition}

For clarity, let us note that this will be combined with Proposition \ref{l1loc} by setting $v_i = u_i - u$, so that $v_i \to 0$, hence $v_i^+ \to 0$, in $L^{1}_{loc}$. Here $u_i$ is one of the two sequences $|\partial_z|_{\omega_{t_i}}^2$ or $\sum_{j=1}^m |dz^j|_{\omega_{t_i}}^2$.

\begin{proof} Let $H_{i,\delta}$ again denote the heat kernel at time $\delta^2$ associated with $(\C^m\times Y, \omega_{t_i})$. As in the proof of Proposition \ref{l1loc}, using the Karp-Li maximum principle \cite[Thm 15.2]{li} and the Li-Yau Gaussian upper heat kernel bounds \cite[Thm 13.4]{li}, we obtain that
\begin{align}
v_i(x) &\leq \int_{\C^m\times Y} H_{i,\delta}(x,x') v_i(x') \omega_{t_i}(x')^{m+n}\notag\\
&\leq C \int_{\C^m\times Y} |B_{\omega_{t_i}}(x,\delta)|_{\omega_{t_i}}^{-1/2}|B_{\omega_{t_i}}(x',\delta)|_{\omega_{t_i}}^{-1/2}\exp(-{\rm dist}_{\omega_{t_i}}(x,x')^2/5\delta^2) v_i^+(x')\omega_{t_i}(x')^{m+n}.\notag
\end{align}
(The Li-Yau bounds produce a constant $C$ independent of $i$ because ${\rm Ric}(\omega_{t_i}) \geq 0$, but it seems very likely that this actually holds without any curvature conditions on $\omega_{t_i}$ or $\omega_Y$; see Remark \ref{rk:grig}.) We now assume that $\delta \leq 1$, aiming to make $\delta$ sufficiently small depending on $v_i$. Then  Bishop-Gromov volume monotonicity yields that 
$|B_{\omega_{t_i}}(\hat{x}, \delta)|_{\omega_{t_i}} \geq C^{-1}e^{-nt_i} \delta^{2(m+n)}$ for $\hat{x} = x,x'$. Moreover assume that $x \in B_1 \times Y$, and decompose the domain $\C^m \times Y$ into $B_2 \times Y$ and its complement. Using our assumption that $v_i^+ \leq C$, it is then straightforward to deduce that
$$\sup_{B_1 \times Y} v_i^+ \leq C\delta^{-2(m+n)}\int_{B_2 \times Y} v_i^+ (\omega_{\C^m}^m \wedge \omega_Y^n) + C \delta^{2-2(m+n)} \exp(-1/C\delta^2).$$
Assuming that $\|v_i^+\|_{L^1(B_2 \times Y)} \leq C^{-1}$, choose $\delta = |C\log \|v_i^+\|_{L^1(B_2 \times Y)}|^{-1/2}$, where the $L^1$ norms are taken with respect to $\omega_{\C^m}^m \wedge \omega_Y^n$. It follows that $\sup_{B_1 \times Y} v_{i}^+ \to 0$ if $\|v_{i}^+\|_{L^1(B_2 \times Y)} \to 0$.
\end{proof}

\subsection{Parallel vector fields imply splitting} The following proposition deals with Step \ref{6}.

\begin{proposition}
Given that $\omega - (\omega_{\C^m} + \omega_Y)$ is $i\partial\bar\partial$-exact and that $\partial_{z^j}$ and $dz^j$ are $\omega$-parallel for all $j \in \{1,\ldots, m\}$ with $|\partial_{z^j}|^2_\omega = \frac{1}{2}$ and $|dz^j|_{\omega}^2 = 2$, it follows that $\omega = \omega_{\C^m} + \omega_Y$.
\end{proposition}

\begin{proof}
Let $\tilde\omega$ be the pullback of $\omega$ to $\C^m \times \tilde{Y}$, where $\tilde{Y}$ is the universal cover of $Y$. Let $\tilde{P}$ be the flat de Rham factor of $(\C^m \times \tilde{Y}, \tilde{\omega})$ spanned by the pullbacks of all $\omega$-parallel vector fields on $\C^m \times Y$. Then $\tilde{P} = \tilde{P}_{\C^m} \oplus \tilde{P}_Y$, where $\tilde{P}_{\C^m}$ is spanned by the pullbacks of $\partial_{z^1}, \ldots, \partial_{z^m}$ and $\tilde{P}_Y$ is the space of lifts of all $\omega_Y$-parallel vector fields on $Y$. (Here we have used the fact that $\nabla^\omega (dz^j) = 0$ for all $j$, which implies that the fibers $\{z\}\times Y$ are $\omega$-totally geodesic and form an $\omega$-parallel family, and also that ${\rm Ric}(\omega|_{\{z\} \times Y}) = 0$, hence $\omega|_{\{z\}\times Y} = \omega_Y$.) The main point to observe is that \begin{small}$\tilde{P}_{\C^m}$\end{small} is not a priori orthogonal to \begin{small}$\tilde{P}_Y$\end{small} with respect to the Euclidean metric $\tilde\omega|_{\tilde{P}}$, but the orthogonal complement of \begin{small}$\tilde{P}_Y$\end{small} is the graph of a unique complex linear map $\ell^\sharp: \tilde{P}_{\C^m} \to \tilde{P}_Y$. Then the splitting $\tilde{P} = \tilde{P}_{\C^m} \oplus \tilde{P}_Y$ is orthogonal with respect to $T_\ell^*\tilde\omega$, and $T_\ell^*\omega = c_\ell \omega_{\C^m} + \omega_Y$ for some constant $c_\ell$. However, since $\omega$ is $i\partial\bar\partial$-cohomologous to $\omega_{\C^m} + \omega_Y$ by assumption, and $c_\ell\omega_{\C^m}$ is $i\partial\bar\partial$-cohomologous to $\omega_{\C^m}$, it follows from the uniqueness statement of Theorem \ref{liouville}(2) that $\ell = 0$.
\end{proof}

This finishes Step \ref{6} and the proof  of Theorem \ref{liouville}. We conclude this paper with a remark on the automorphisms $T_\ell$. After reading the statement of Theorem \ref{liouville}(1) one may be tempted to guess that $\omega$ is always equal to a product K\"ahler form $S^*\omega_{\C^m} + \omega_Y$. We have proved that this is  true if one replaces $\omega$ by $T_\ell^*\omega$ for a uniquely defined complex linear map $\ell: \C^m \to H^{0,1}(Y)$, where $T_\ell = {\rm id}$ if and only if $\ell = 0$. This still leaves the possibility that $\ell$ might always vanish. Now Examples \ref{bad_example}--\ref{bad_example_contd} show that $\ell \neq 0$ in general, and that this obstructs not only the conclusion that $\omega$ is a product form but also a key technical step of its proof (the statement of Proposition \ref{weaktrace}). Our final remark says that \emph{every} nonzero linear map $\ell$ can be used to generate a counterexample in this way.

\begin{rk}
Given a complex linear map $\ell: \C^m \to H^{0,1}(Y)$, the K\"ahler metric $T_\ell^*(\omega_{\C^m} + \omega_Y)$ is Ricci-flat with the same volume form as the product metric, and is parallel with respect to it. The vector fields $\partial_{z^j}$ are still parallel with respect to $T_\ell^*(\omega_{\C^m} + \omega_Y)$, but they do not define an isometric product splitting unless $\ell = 0$; instead one needs to use the parallel vector fields $(\partial_{z^j}, - \frac{\partial}{\partial z^j} \ell^\sharp)$.
\end{rk}

 \end{document}